\documentclass[a4paper,11pt,reqno]{amsart}

\usepackage{amsfonts}
\usepackage{amssymb}
\usepackage{amscd}
\usepackage{amsthm}
\usepackage{a4wide}
\usepackage{mathrsfs}

\usepackage[centertags]{amsmath}
\usepackage{eucal,dsfont,accents,bbm,amsfonts,amssymb,amsthm,amsopn}

\usepackage[usenames]{color}
\definecolor{citegreen}{rgb}{0,0.6,0}
\definecolor{refred}{rgb}{0.8,0,0}
\usepackage[colorlinks, citecolor=citegreen, linkcolor=refred]{hyperref}

\theoremstyle{plain}

\newtheorem{teo}{Theorem}[section]
\newtheorem{lemma}[teo]{Lemma}
\newtheorem{prop}[teo]{Proposition}
\newtheorem{cor}[teo]{Corollary}
\newtheorem{ackn}{Acknowledgments\!}

\theoremstyle{definition}

\theoremstyle{remark}

\numberwithin{equation}{section}

\def\RR{{\mathbb R}}

\def\RRR{{\mathrm R}}

\def\Ric{{\mathrm {Ric}}}

\def\Scal{{\RRR}}

\def\Rm{{\mathrm {Rm}}}
\def\de{{\partial}}

\def\a{\alpha}

\def\eps{\varepsilon}

\def\tr{\operatornamewithlimits{tr}\nolimits}
\def\dist{\mathrm{dist}}

\title[Rigidity of gradient Einstein shrinkers]{Rigidity of gradient Einstein shrinkers}

\author[Giovanni Catino]{Giovanni Catino}
\address[Giovanni Catino]{Dipartimento di Matematica, Politecnico di Milano, Piazza Leonardo da Vinci 32, 20133 Milano, Italy}
\email[]{giovanni.catino@polimi.it}

\author[Lorenzo Mazzieri]{Lorenzo Mazzieri}
\address[Lorenzo Mazzieri]{Scuola Normale Superiore, P.za Cavalieri 7, 56126 Pisa, Italy}
\email{l.mazzieri@sns.it}

\author[Samuele Mongodi]{Samuele Mongodi}
\address[Samuele Mongodi]{Scuola Normale Superiore, P.za Cavalieri 7, 56126 Pisa, Italy}
\email{s.mongodi@sns.it}

\begin{document}

\begin{abstract} In this paper we consider a perturbation of the Ricci solitons equation proposed in~\cite{jpb1} and studied in~\cite{CaMa} and we classify noncompact gradient shrinkers with bounded nonnegative sectional curvature.
\end{abstract}

\maketitle

\begin{center}


\medskip

\centerline{\bf AMS subject classification:  53C24, 53C25, 53C44}

\end{center}

\

\section{Introduction and statement of the result}

In recent years much effort has been devoted to the classification of self-similar solutions of geometric flows. Some of the most interesting examples are gradient Ricci solitons . These are Riemannian manifolds satisfying 
$$
\Ric + \nabla^{2}f \,=\, \lambda g \,,
$$
for some $\lambda\in\mathbb{R}$ and some smooth function $f$ defined on $M^{n}$. In particular, if $\lambda>0$, the soliton is called {\em shrinking} and it generates an ancient self-similar solution to the Ricci flow with finite extinction time. In dimension three, a complete classification of gradient Ricci shrinkers was given by Ivey~\cite{ivey1} in the compact case and by Perelman~\cite{perel1}, Ni-Wallach~\cite{nw2} and Cao-Chen-Zhu~\cite{caochenzhu} in the complete case. In higher dimension the situation is much more complicated. In fact, due to the lack of the Hamilton-Ivey pinching estimates, which ensure the nonnegativity of the sectional curvatures in dimension three, there exist examples of ``exotic'' shrinking Ricci solitons in both the compact and the noncompact case (for a general overview on Ricci shrinkers, we refer the reader to~\cite{cao3}). In dimension four, under the assumption of bounded nonnegative curvature operator, the most significant classification result has been obtained by Naber~\cite{na1}, where he proves that any noncompact Ricci shrinker of this type is isometric to $\mathbb{R}^{4}$ or to a finite quotient of either $\mathbb{S}^{2}\times\mathbb{R}^{2}$ or $\mathbb{S}^{3}\times\mathbb{R}$. In higher dimension we would like to mention the following result due to Petersen-Wylie~\cite{pw2}.
\begin{teo}[Petersen-Wylie~\cite{pw2}] 
\label{petwyl}
A complete, noncompact, rectifiable, gradient shrinking Ricci soliton with bounded curvature, nonnegative radial sectional curvature, and nonnegative Ricci curvature is rigid.
\end{teo}  
To understand the statement, we recall that a soliton is called \emph{rectifiable} if $|\nabla f|$ is constant along the connected components of the regular level sets of $f$ and it is called {\em rigid} if, for some $k\in \{0,\ldots,(n-1) \}$, its universal cover, endowed with the lifted metric and the lifted potential function, is isometric to the Riemannian product $N^{k}\times\mathbb{R}^{n-k}$, where $N^{k}$ is a $k$-dimensional Einstein manifold and $f = \tfrac{\lambda}{2} |x|^{2}$ on the Euclidean factor. We also recall that $g$ has nonnegative radial sectional curvature if $\Rm(E,\nabla f,E,\nabla f)\geq 0$ for every vector field $E$ orthogonal to $\nabla f$. 

In this paper we consider the following perturbation of the Ricci soliton equation 
\begin{equation}\label{eq_soliton}
\Ric + \nabla^2 f=\rho \Scal g + \lambda g \, ,
\end{equation}
where $(M^{n},g)$ is a Riemannian manifold, $\lambda\in\mathbb{R}$, $\rho\in\mathbb{R}\setminus\{0\}$ and $f$ is a smooth function on $M^{n}$ which will be called \emph{potential}. Solutions to this equation are called gradient $\rho$-Einstein solitons and were first considered in~\cite{CaMa}, where various classification results have been obtained, in particular in the steady case $\lambda=0$.

As in the case of Ricci solitons, it is easy to see that $\rho$-Einstein solitons give rise to self-similar solutions to a perturbed version of the Ricci flow, the so called Ricci-Bourguignon flow
\begin{equation*}
\frac{\partial}{\partial t} g=-2(\Ric- \rho\Scal g)\,.\end{equation*}
In a forthcoming paper, we will develop the parabolic theory for these flows, which was first considered by Bourguignon in~\cite{jpb1}. Here we just remark that we can prove short time existence for every $-\infty<\rho< 1/2(n-1)$. However, as far as the subject of our investigation are self-similar solutions, every value of $\rho$ can, in principle, be considered. In particular, we point out that the case $\rho=1/2(n-1)$ corresponds to a metric flowing with velocity proportional to its Schouten tensor. In this case, it was proved in~\cite{CaMa} that every three-dimensional Schouten shrinker is rigid.

The main result of the present paper is the following theorem.
\begin{teo}\label{main} Let $(M^{n},g)$, with $n \geq 3$, be a complete, noncompact, gradient shrinking $\rho$-Einstein soliton with $0<\rho< 1/2(n-1)$. If $g$ has bounded curvature, nonnegative radial sectional curvature, and nonnegative Ricci curvature, then $(M^{n},g)$ is rigid.
\end{teo}
As it is evident, the statement is the precise analogous of the aforementioned result of Petersen and Wylie. We emphasize the remarkable fact that, in our case, we do not need any  symmetry assumption. In fact, the rectifiability can be deduced from the structural equation~\eqref{eq_soliton}, as it is proved in~\cite{CaMa}.

\

\section{Preliminaries on gradient $\rho$-Einstein solitons}

First of all, we show that gradient $\rho$-Einstein solitons give rise to solutions of the Ricci-Bourguignon flow
\begin{equation}\label{eq_Einsteinflow}\frac{\partial}{\partial t} g=-2(\Ric- \rho\Scal g)\;.\end{equation}
Although the proof is quite similar to the classical one for Ricci solitons, we include it for the convenience of the reader. 

By a \emph{complete} gradient $\rho$-Einstein soliton, we mean a complete Riemannian manifold $(M,g)$ with a potential $f$ such that $\nabla^g f$ is complete and \eqref{eq_soliton} holds.

\begin{teo}\label{teo_solflusso}If $(M,g_0,f_0)$ is a complete gradient $\rho$-Einstein soliton with constant $\lambda$, then there exist 
\renewcommand{\labelenumi}{\roman{enumi}. }
\begin{enumerate}
\item a family of metrics $g(t)$, solution of the Ricci-Bourguignon flow \eqref{eq_Einsteinflow}, with $g(0)=g_0$,
\item a family of diffeomorphisms $\phi(t , \cdot \,):M\to M$, with $\phi(0, \cdot \,)=\mathrm{id}_M$,
\item a family of functions $f(t, \cdot \,):M\to \RR$ with $f(0, \cdot	\,)=f_0(\cdot)$,
\end{enumerate}
defined for every $t$ such that $\tau(t):=-2\lambda  t+1>0$. These families have the following properties:
\renewcommand{\labelenumi}{\arabic{enumi}. }
\begin{enumerate}
\item the family $\phi(t, \cdot \,)$ is generated by the vector field $\nabla^{g_0}f_0$ eventually scaled by the inverse of $\tau(t)$
\begin{equation}
\label{eq_defdiffeo}
\frac{\partial \phi}{\partial t}(t,\cdot\,) \, = \, \frac{1}{\tau(t)}(\nabla^{g_0}f_0)(\phi(t,\cdot\,))\;,
\end{equation}

\item the metric $g(t)$ is given by pull-back through $\phi(t,\cdot\,)$ and rescaling through $\tau(t)$
\begin{equation}
\label{eq_defmetr}g(t) \, = \, \tau(t)\, \phi(t,\cdot)^*g_0 \, ,
\end{equation}
\item the function $f(t)$ is given as well by pull-back, namely
\begin{equation}
\label{eq_deff}
f(t, \cdot) \, = \, (f_0\circ\phi)(t, \cdot\,)\;.\end{equation}
\end{enumerate}
\end{teo}
\begin{proof} We set $\tau(t)=-2\lambda t+1$. As $\nabla^{g_0}f_0$ is a complete vector-field, there exists a $1$-parameter family of diffeomorphisms $\phi(t,\cdot\,):M\to M$ generated by the time dependent family of vector fiels $X(t, \cdot \,) \, := \,\frac{1}{\tau(t)} \, \nabla^{g_0}f_0 (\phi(t,\cdot \,))$, for every $t$ such that $\tau(t)>0$. We also set $f(t, \cdot)= (f_0\circ\phi)(t, \cdot\,)$ and $g(t)=\tau(t)\, \phi(t)^*g_0$. We compute
$$
\frac{\partial}{\partial t}g(t) \, = \, -\frac{2\lambda}{\tau(t)}g(t) \, + \, \tau(t)\frac{\partial}{\partial t} \phi(t,\cdot)^*g_0\;.
$$
By the definition of the Lie derivative, we have that
$
\frac{\partial}{\partial t}\phi(t, \cdot)^*g_0={\mathscr{L}}_{(\phi(t)^{-1})_*\frac{\partial}{\partial t}\phi(t,\cdot)}\phi(t,\cdot)^*g_0 \, .
$
On the other hand, equation \eqref{eq_defdiffeo} implies that
$$
\frac{\partial \phi}{\partial t}(t,\cdot\,) \, = \, \frac{1}{\tau(t)}(\nabla^{g_0}f_0) (\cdot)\, = \, \frac{1}{\tau(t)}\phi(t,\cdot\,)_*\nabla^{g(t)}f(t, \cdot\,) \, ,
$$
where we used the fact that $\phi(t, \cdot)^*\nabla^{g_0}f_0=\nabla^{\phi(t, \cdot)^*g_0}\phi(t, \cdot)^*f_0 \, = \, \nabla^{g(t)}f(t, \cdot)$. 
Combining these two facts, we have that
$$
\frac{\partial}{\partial t}g(t) \, = \, -\frac{2\lambda}{\tau(t)}g(t)+\frac{1}{\tau(t)}\mathscr{L}_{\nabla^{g(t)}f(t,\cdot)}g(t)\;.
$$
Having this at hand, we compute
\begin{eqnarray*}
-\Ric(g(t)) & = & \phi(t, \cdot)^*(- \, \Ric(g_0)) \,\, = \,\, \phi(t,\cdot)^*  \bigg( \,  \frac{1}{2}\mathscr{L}_{\nabla^{g_0}f_0}g_0 \, - \, \lambda \, g_0 \, - \, \rho\Scal(g_0) \, g_0 \, \bigg)   \\
& = & \frac{1}{2}\left(\frac{1}{\tau(t)} \mathscr{L}_{\nabla^{g(t)}f(t, \cdot)}g(t) \, - \, \frac{2}{\tau(t)}\lambda \,  g(t)\right) \, - \, \frac{\rho}{\tau(t)} \, \Scal(\tau(t)^{-1}g(t)) \, g(t)\\
& = & \frac{1}{2}\frac{\partial}{\partial t}g(t) \, - \, \frac{\rho}{\tau(t)} \, \Scal(\tau(t)^{-1}g(t)) \, g(t)
\end{eqnarray*}
and we observe that $\Scal(\tau(t)^{-1}g(t)) \, = \,\tau(t)\, \Scal(g(t))$. In other words, we have obtained
$$
\frac{\partial}{\partial t}g(t) \, = \, -2 \,[ \,  \Ric(g(t)) \, - \, \rho\Scal(g(t)) \, g(t) \, ] \;,
$$
and the proof is complete.
\end{proof}

In particular, we have obtained that shrinking solitons generate ancient solutions, which blow up at $t=1/2\lambda$.

We pass now to describe a fundamental property of the gradient $\rho$-Einstein solitons, namely the rectifiability.
To do that, we recall from~\cite[Lemma 2.2 and Theorem 3.1]{CaMa} the following fundamental identities for the gradient $\rho$-Einstein solitons. We also report the proof, for the convenience of the reader.
\begin{lemma}
\label{lemg} 
Let $(M^{n},g,f)$, $n\geq 3$, be a gradient $\rho$-Einstein soliton. Then, the following identities hold true.
\begin{eqnarray}
\Delta f & = &  (n\rho-1)\RRR+n\lambda\,,\label{equa1}\\
\big(1-2(n-1)\rho\big)\nabla \RRR & = &  2 \Ric (\nabla f, \,\cdot\,)\,,\label{equa2}\\
\label{equa3}
\big(1-2(n-1)\rho\big)\Delta \RRR & = &  \langle \nabla \RRR,\nabla f\rangle + 2( \rho \RRR^{2}-|\Ric|^{2}+\lambda \RRR )\,, \\
\label{equa4} d\RRR \otimes df & = &  d f \otimes d\RRR\,.
\end{eqnarray}
\end{lemma}

\begin{proof} Taking the trace of equation~\eqref{eq_soliton}, we obtain
$$
\Scal + \Delta f \,  =  \, n\rho\Scal+\lambda n \, ,
$$
which is equation~\eqref{equa1}. Taking the divergence of equation~\eqref{eq_soliton}, we obtain
$$
\nabla_i\RRR_{ij} + \nabla_i\nabla_i\nabla_j f \, = \, \rho\nabla_i\Scal g_{ij}\;.
$$
Using the formula for the commutation of the derivatives, we get
$$
\frac{1}{2}\nabla_j\Scal + \nabla_j\nabla_i\Delta f+\RRR_{ijip}\nabla_p f \, = \, \rho\nabla_j\Scal \,.
$$
Up to rearranging the terms, this is equivalent to
$$
\left(\frac{1}{2}-\rho\right)\nabla_j\Scal + \nabla_j\Delta f +\RRR_{jp}\nabla_p f \, = \, 0\;.
$$
If we substitute equation~\eqref{equa1}, in the identity above, we arrive to
$$
(1-2(n-1)\rho) \, \nabla\Scal \, = \, 2 \, \Ric(\nabla f,\, \cdot \, )\;,
$$
that is, equation~\eqref{equa2}. If we take the divergence of equation~\eqref{equa2}, we get
\begin{eqnarray*}
(1-2(n-1)\rho) \, \Delta\Scal & = & 2\nabla_i\RRR_{ip}\nabla_p f +2\RRR_{ip}\nabla_i\nabla_p f\\
& = & \nabla_p\Scal \nabla_p f + 2\RRR_{ip}(\rho \Scal g_{ip}+\lambda g_{ip} - \RRR_{ip})\\
&=&\langle \nabla\Scal \, | \,  \nabla f\rangle+2(\rho \Scal^2 + \lambda \Scal - |\Ric|^2)\;.
\end{eqnarray*}
This proves equation~\eqref{equa3}. Taking the covariant derivative of equation~\eqref{equa2}, we obtain
$$
(1-2(n-1)\rho) \, \nabla_i\nabla_j\Scal \, = \, 2\nabla_i\RRR_{jp}\nabla_p f + 2\RRR_{jp}\nabla_i\nabla_p f\;.
$$
By the symmetry of the Hessian, we deduce that
\begin{eqnarray*}
0 & = & (1-2(n-1)\rho) \, (\nabla_i\nabla_j \RRR \, -\, \nabla_j\nabla_i\Scal ) \\  
& = & \, 2 \, (\nabla_i\RRR_{jp}-\nabla_j\RRR_{ip})\, \nabla_p f \,  + \,  2 \, (\RRR_{jp}\nabla_i\nabla_p f- \RRR_{ip}\nabla_j\nabla_p f)\;.
\end{eqnarray*}
Taking the covariant derivative of equation (\ref{eq_soliton}) and rotating indices, we infer that
\begin{eqnarray*}
\nabla_i\RRR_{jp}-\nabla_j\RRR_{ip} & = & \nabla_j\nabla_i\nabla_pf -\nabla_i\nabla_j\nabla_pf+\rho\nabla_i\Scal \, g_{jp}-\rho\nabla_j\Scal \, g_{ip}\\
&=&\RRR_{jipk}\nabla_k f+\rho \, (\nabla_i\Scal \, g_{jp}-\nabla_j\Scal \, g_{ip}) \, .
\end{eqnarray*}
Substituting this expression in the previous formula, we get
$$0=2\RRR_{jipk}\nabla_k f\nabla_p f+2\rho(\nabla_i\Scal\nabla_j f -\nabla_j\Scal\nabla_i f)+2(\RRR_{jp}\nabla_i\nabla_p f- \RRR_{ip}\nabla_j\nabla_p f)\;.$$
We note that $\RRR_{jipk}\nabla_kf\nabla_p f=0$, as the curvature tensor is antisymmetric in the last two indices while $\nabla_k f\nabla_p f$ is symmetric. Again from equation (\ref{eq_soliton}), we obtain that
$$
\RRR_{jp}\nabla_i\nabla_q f \, = \, (\rho\Scal+ \lambda) \RRR_{jp}g_{iq}-\RRR_{jp}\RRR_{iq}
$$
and thus
$$
2(\RRR_{jp}\nabla_i\nabla_p f- \RRR_{ip}\nabla_j\nabla_p f) \, = \, 2(\rho\Scal+\lambda)(\RRR_{ij}-\RRR_{ij}) \, = \, 0\;.
$$
Substituting again, we finally get
$$
0 \, = \, \rho \, (\nabla_i\Scal\nabla_j f -\nabla_j\Scal\nabla_i f) \,.
$$
Using the fact that $\rho\neq0$, we deduce equation (\ref{equa4}) and the lemma is proved.
\end{proof}

Following~\cite{CaMa}, we notice that whenever $|\nabla f| \neq 0$ the gradient of the scalar curvature $\nabla \RRR$ is proportional to $\nabla f$. In fact, if $p \in M$ is a point such that $\nabla f (p) \neq 0$, we let $V \in T_p M$ be any vector which is orthogonal to $\nabla f$. By equation~\eqref{equa4}, we get
\begin{equation}\label{rcost}
\langle \nabla \RRR \, | \, V\rangle \,\,|\nabla f|^{2} \, = \, \langle \nabla \RRR \, | \, \nabla f\rangle \, \langle\nabla f \, | \, V\rangle\,=\,0\,,
\end{equation}
and hence $\langle \nabla \RRR \, | \, V \rangle \, = \,0$ at $p$. From this we deduce that the same is true for $\nabla |\nabla f|$. In fact, from the structural equation~\eqref{eq_soliton}, we infer that
\begin{eqnarray}\label{gfcost}
\langle \nabla|\nabla f|^{2} \, | \, V\rangle & = & 2 \,\nabla^{2} f \, (\nabla f, V)\\\nonumber
&=& (2\rho \RRR + 2 \lambda) \, \langle \nabla f \, | \, V\rangle-{2}\Ric\, (\nabla f,V) \\\nonumber
&=& -({1-2(n-1)\rho}) \, \langle \nabla \RRR \, | \, V \rangle \,\,=\,\,0\,,
\end{eqnarray}
where in the last equality we have used equation~\eqref{equa2}. In particular, we have obtained the following theorem.
\begin{teo}[\cite{CaMa} Catino-Mazzieri]
\label{teorect}
Every gradient $\rho$-Einstein soliton is rectifiable.
\end{teo}


Now, we turn our attention to the regularity of gradient $\rho$-Einstein solitons. We recall that, in harmonic coordinates, one has
\begin{equation}
\label{eq_RicHarm}
\Ric \, = \, -\frac{1}{2}\Delta (g_{ij}) + Q_{ij}(g^{-1}, \partial g) \, ,
\end{equation}
where $Q$ is a quadratic form in the coefficients of $g^{-1}$ and the first derivatives of the coefficients of $g$.

\begin{teo}
\label{teo_analytic}
Let $(M^{n},g,f)$, $n\geq 3$, be a gradient $\rho$-Einstein soliton, with $\rho \notin \{1/n,1/2(n-1)\}$. Then, in harmonic coordinates, the metric $g$ and the potential function $f$ are real analytic.
\end{teo}
\begin{proof} We note that taking the divergence of equation~\eqref{eq_soliton} we get
$$
\nabla_j\nabla_j\nabla_i f \, = \, -\nabla_j\RRR_{ij}+\rho\nabla_i \Scal g_{ij}=-\frac{1}{2}\nabla_i\Scal + \rho\nabla_i\Scal \, .
$$
Thus, using equation~\eqref{equa2}, we obtain
$$-\Delta \nabla f \, = \, \frac{1-2\rho}{1-2(n-1)\rho}\Ric(\nabla f, \cdot)\;.
$$
To prove our statement, it is useful to consider the system
$$
\left\{\begin{array}{rcl}
\Ric \,+ \, \nabla^2 f \, - \, \lambda\, g \, - \, \rho \,\Scal \,g & = & 0 \\
-\Delta \nabla f \,- \,\Big( \frac{1-2\rho}{1-2(n-1)\rho} \Big)\,\, \Ric(\nabla f, \,\cdot \,) & = & 0 \, ,\end{array} \, \right. 
$$
with respect to the unknowns $(g,\nabla f)$. According to~\eqref{eq_RicHarm}, we have that in harmonic coordinates the scalar curvature is given by
$$
\Scal \, = \, -\frac{1}{2}g^{ij}\Delta(g_{ij}) + g^{ij}Q_{ij}(g,\partial g) \, .
$$
Thus, the linearization of the previous system in the direction of $(h,W)\in S^2T^*M\oplus TM$ is given by
$$
\displaystyle{\left\{\begin{array}{rcl}
-\dfrac{1}{2}g^{rs}\dfrac{\de^2 h_{ij}}{\de x^r\de x^s}+\dfrac{\rho}{2}g^{kl}g^{rs}\dfrac{\de^2 h_{kl}}{\de x^r \de x^s}g_{ij}+\textrm{l.o.t.}&=&0\\
 & & \\
-g^{rs}\dfrac{\de^2 W_i}{\de x^r\de x^s}+\dfrac{1-2\rho}{2-4(n-1)\rho}g^{rs}\dfrac{\de^2 h_{ij}}{\de x^r\de x^s}\nabla_j f+\textrm{l.o.t.}&=&0 \, ,\end{array}\right.}
$$
where l.o.t denotes terms involving only $W$, $h$ or their first derivatives. Therefore, the principal symbol $\sigma_{\zeta}:S^2T^*M\oplus TM\to S^2T^*M\oplus TM$ is given by
$$
(h, W) \, \longmapsto \, \sigma_\zeta (h, W) \, = \,  \bigg( \, \frac{1}2|\zeta|^2_{g}(h-\rho (\mathrm{tr}_{g}h) g) \, , \,  |\zeta|^2_{g}W -L_\zeta h \, \bigg) \, ,
$$
where $L_\zeta h$ is some linear function of $h$. If $\sigma_\zeta(h,W)=0$ and $\zeta \neq 0$, then $h=\rho (\mathrm{tr}_g h) g$ and thus
$$
\mathrm{tr}_g h \, = \, \rho\, n \,\tr_gh \, ,
$$
that is $\mathrm{tr}_gh=0$ or $\rho n=1$. The latter implies that $\rho=1/n$, which is excluded by our hypothesis, whereas the former gives $h=0$, since $h=\rho (\mathrm{tr}_g h) g$ and by definition $\rho \neq 0$. Consequently, if $\sigma_\zeta(h,W)=0$, then we must have $h=0$, which implies $W=0$. This shows that, if $\rho\neq1/n$ and $\zeta \neq 0$, the symbol $\sigma_\zeta$ is an automorphism of $S^2T^*M\oplus TM$ and this in turn implies that the linearization of the system is elliptic.

If $(g,\nabla f)$ have ${C}^2$-regularity in harmonic coordinates, we can apply Morrey's interior regularity theorem \cite[Theorem 6.7.6]{Morr} and since our system of equations is analytic in both its dependent and independent variables, the solutions are real analytic as well.
We observe that in general $(g,\nabla f)$ could be only ${C}^{1,\alpha}$ after passing to harmonic coordinates. To overcome this difficulty, we apply Theorem 9.19 in \cite{GiTr} to the components of the system, to obtain that $(g,\nabla f)$ are in fact ${C}^{2,\alpha}$. 
\end{proof}

\

\section{Estimates on the growth of the potential function}

In this section we consider shrinking solitons with bounded non-negative scalar curvature, namely $\RRR\geq 0$ and $|\RRR|\leq K$, for some positive constant $K$. 
To proceed, we observe that, in force of the rectifiability of the gradient $\rho$-Einstein solitons (see Theorem~\ref{teorect}), either $f$ is constant on $M$ or there exists a hypersurface $\Sigma_0 \subset M$, which is a connected component of a regular level set of $f$. We also recall that, in the latter case, the potential function $f$, as well as the scalar curvature $\RRR$ and the function $|\nabla f|$, only depends on the signed distance $r$ to $\Sigma_0$, {\em a priori} only in a suitable neighborhood of it. On the other hand, since $f$ is real analytic, we have that the level sets where $|\nabla f|= 0$ cannot accumulate, unless $f$ is constant. Hence, as soon as a regular level set exists, we have that $f$ only depends on the signed distance to $\Sigma_0$ on the whole manifold. Of course, the same is true for $\RRR$ and $|\nabla f|$.  

The goal of this section is to prove that either $f$ is constant, or we have an estimate of the following type
$$
A\, (\,|r|+B\, )^2 \, \geq \, f(r) \, \geq \, C\,(\,|r| - D\,)^2 \, ,
$$
where $A,B,C$ and $D$ are positive real constants.
We start with the following lemma, whose proof is not as direct as in the Ricci soliton case, due to the lack of the Hamilton's identity.
\begin{lemma}
\label{lmm_grad_lin} 
Let $(M^n, g,f)$ be a gradient shrinking $\rho$-Einstein soliton with $\rho>0$, $\RRR\geq 0$ and such that $|\RRR|\leq K$, for some positive constant $K$. Then, either $f$ is constant or there exist positive real constants $a^\pm$, $b^\pm$, $c^\pm$ and $d^\pm$, such that 
\begin{eqnarray*}
c^+f(r) - d^+ \, \leq \, |\nabla f|^2(r) \, \leq \, a^+f(r) + b^+\;, & & \hbox {for\quad}r\geq 0, \\
c^-f(r) - d^- \, \leq \, |\nabla f|^2(r) \, \leq \, a^-f(r) + b^-\;, & & \hbox {for\quad}r\leq 0,
\end{eqnarray*}
where $r$ is the signed distance to a connected component $\Sigma_0\subset M$ of some regular level set of $f$. The constants which appear in the estimates are possibly depending on $\Sigma_0$.
\end{lemma}
\begin{proof} Let us assume that $f$ is not constant. Then, there exist a point $p_0$, such that $|\nabla f|(p_0)>0$. We now let $\Sigma_0$ be the connected component of the level set $\{ f= f(p_0)\}$, which contains the point $p_0$. By Theorem~\ref{teorect}, we have that $|\nabla f|$ is constant along $\Sigma_0$. Therefore, $\Sigma_0$ is regular. According to the discussion above, we let $r$ be the signed distance to $\Sigma_0$. The orientation of $r$ is the one which insures $\langle \nabla f \, | \nabla r \rangle \geq 0$ around $\Sigma_0$.

We consider now the function $a^+ f-|\nabla f|^2$, with $a^+ >0$. If we compute its derivative along $\nabla f$, we get
\begin{eqnarray*}
\langle \nabla(a^+ f-|\nabla f|^2)\vert \nabla f\rangle & = & a^+ |\nabla f|^2-2\nabla^2 f(\nabla f, \nabla f)\\
&=& a^+ |\nabla f|^2 + 2\Ric(\nabla f, \nabla f)- 2\rho \Scal|\nabla f|^2- 2\lambda |\nabla f|^2\\
&=& (a^+ -2\lambda -2\rho \Scal)|\nabla f|^2 + (1-2(n-1)\rho)\langle \nabla \Scal\vert \nabla f\rangle \, ,
\end{eqnarray*}
where the last equality was obtained by equation~\eqref{equa2} together with~\eqref{eq_soliton}. Therefore, one has
$$
\langle \nabla(a^+ f - |\nabla f|^2 - (1-2(n-1)\rho)\Scal)\, \vert \, \nabla f\rangle \, =  \, (a^+ -2\lambda -2\rho \Scal)|\nabla f|^2\;.
$$
If $\Scal \leq K$, for some real constant $K$, then it is enough to choose $a^+ > 2\lambda+2\rho K$ to obtain that the function $ \Phi = a^+ f-|\nabla f|^2-(1-2(n-1))\rho \Scal$ is increasing in the direction of $\nabla f$. We let now $q$ be a point in $M$ such that there exists an integral curve $\gamma \, : \, [0, L] \rightarrow M$ of $\nabla f$ with $\gamma(0) \in \Sigma_0$ and $\gamma (L) = q$. Integrating the function $\Phi \circ \gamma$ on $[0,L]$ and using the computation above, it is immediate to see that $\Phi(q) \geq  \Phi (\gamma(0))$. With a small abuse of notation, we can consider $\Phi$ as a function of $r$ and the last inequality can be written as $\Phi(r) \geq \Phi (0)$, for every $r\geq 0$. By the definition of $\Phi$, we obtain, for every $r\geq 0$, the estimate
%
$$
|\nabla f|^2 (r) \, \leq \,  a^+ f(r)-(1-2(n-1)\rho)\Scal - \Phi(0)  \, \leq \, a^+ f(r) + b^+ \, ,
$$
where we used the fact that $|\RRR|<K$ and we set $b^+= |(1-2(n-1)\rho)K|  + |\Phi(0)|$.
To proceed, we consider now the function $|\nabla f|^2- c^+ f$, with $c^+>0$, and we compute its radial derivative, namely
\begin{eqnarray*}
\langle \nabla(|\nabla f|^2- c^+ f)\vert \nabla f\rangle&=&2\nabla^2 f(\nabla f, \nabla f) - c^+ |\nabla f|^2\\
&=&-2\Ric(\nabla f, \nabla f) + (2\lambda - c^+) |\nabla f|^2 + 2\rho \Scal|\nabla f|^2\\
& \geq &-(1-2(n-1)\rho)\langle \nabla \Scal, \nabla f\rangle+2\rho\Scal|\nabla f|^2 \, ,
\end{eqnarray*}
provided $c^+\leq 2\lambda$. Therefore, we have that 
$$
\langle \nabla(|\nabla f|^2 - c^+ f+(1-2(n-1)\rho)\Scal) \, \vert \, \nabla f\rangle \, = \, 2\rho \, \Scal \, | \nabla f|^2 \, \geq \, 0 \, ,
$$
since $\Scal\geq0$ and $\rho\geq 0$. Reasoning as before, we set now $\Psi = |\nabla f|^2 - c^+ f+(1-2(n-1)\rho)\Scal $ and we get $\Psi(r) \geq \Psi(0)$, for every $r\geq 0$.
In other words, since $\Scal$ is bounded, there exists a positive constant $d^+$, possibly depending on $\rho$, $\Psi(0)$ and the scalar curvature bound $K$, such that, for every $r\geq0$, the following inequality holds
$$
|\nabla f|^2 (r) \, \geq \, \Psi(0) + c^+ f(r) - (1-2(n-1)\rho)\Scal(r) \, \geq \, c^+ f(r) - d^+ \, .
$$
So far, we have obtained the desired estimate provided $r\geq 0$, namely
$$
c^+f(r) - d^+ \, \leq \, |\nabla f|^2(r) \, \leq \, a^+f(r) + b^+\;,
$$
To obtain the analogous estimates, in the case $r\leq 0$, it is sufficient to compute the derivatives of the functions $a^- f - |\nabla f|^2$ and $|\nabla f|^2 - c^- f$ along the vector field $-\nabla f$, and to check that it is possible to choose the positive constants $a^-,b^- ,c^-$ and $d^-$ in such a way that 
$$
c^-f(r) - d^- \, \leq \, |\nabla f|^2(r) \, \leq \, a^-f(r) + b^-\;.
$$
Since the reasoning is the same as in the the case $r\geq 0$, we left the details to the reader. This concludes the proof of the lemma.
\end{proof}

We now proceed with another lemma, which contains an estimate on the lower bound for the potential function $f$. We will employ a slight variation of the method exposed in~\cite{caozhou}.

\begin{lemma}
\label{lmm_stima2} 
Let $(M^n, g,f)$ be a gradient shrinking $\rho$-Einstein soliton with $\rho>0$, $\RRR\geq 0$ and such that $|\RRR|\leq K$, for some positive constant $K$. Then, either $f$ is constant on $M$ or there exist positive constants $C$ and $D$, such that 
$$
f (r) \, \geq \, C ( \,|r| - D \,)^2 \, ,
$$
where $r$ is the signed distance to a connected component $\Sigma_0\subset M$ of some regular level set of $f$. The constants which appear in the estimate are possibly depending on $\Sigma_0$.
\end{lemma}
\begin{proof} From Lemma \ref{lmm_grad_lin}, we have that
$$0\leq |\nabla f|^2\leq a^{\pm}f+b^{\pm}$$
so, considering $g(r)=\sqrt{a^{\pm}f(r)+b^{\pm}}$, we have
\begin{equation}\label{eq_gradlip}|\nabla g(r)|=\dfrac{a^{\pm}|\nabla f|(r)}{2g(r)}\leq\dfrac{a^{\pm}}{2}\;.\end{equation}

Now, let $p,q\in M$ and let $\gamma$ be a minimizing geodesic between them, such that $\gamma(0)=p$ and $\gamma(s_0)=q$, with $s_0=\dist_g(p,q)>2$. We set now
$$
\phi(s) \, = \, \left\{\begin{array}{ll}s,&s\in[0,1]\\
1, &s\in[1, s_0-1]\\s_0-s,&s\in[s_0-1,s_0]\;.
\end{array}\right.
$$
By the second variation formula for the energy of $\gamma$, we have 
\begin{equation}
\label{eq_snd_var}
\int_{0}^{s_0}\phi^2 \, \Ric(\dot{\gamma},\dot{\gamma}) \, ds\,\leq \, (n-1)\int_{0}^{s_0} (\dot\phi)^2 \, ds \, = \, 2n-2\end{equation}
and, by the soliton equation~\eqref{eq_soliton}, we get
$$
\Ric(\dot{\gamma},\dot{\gamma}) \, = \, \lambda|\dot{\gamma}|^2 +\rho\Scal|\dot{\gamma}|^2-\nabla^2f(\dot{\gamma},\dot{\gamma}) \, = \, \lambda+\rho\Scal-\nabla_{\dot{\gamma}}\nabla_{\dot{\gamma}}f\;.
$$
Therefore, we can write
$$
\int_0^{s_0}\phi^2\Ric(\dot{\gamma},\dot{\gamma}) \, ds \, =\, \lambda\int_0^{s_0}\phi^2 \, ds  \, + \,  \rho\int_0^{s_0}\Scal\phi^2 \, ds \, - \, \int_0^{s_0}\phi^2\nabla_{\dot{\gamma}}\nabla_{\dot{\gamma}} f \, ds \, .
$$
Integrating by parts the last term of the right hand side, we get
$$
\int_0^{s_0}\phi^2\nabla_{\dot{\gamma}}\nabla_{\dot{\gamma}} f \, ds  \, = \, -\, 2\int_0^1\phi \, \nabla_{\dot{\gamma}} f\, ds \, + \, 2\int_{s_0-1}^{s_0}\phi \, \nabla_{\dot{\gamma}} f \, ds \, .
$$
The contribution of the interval $[1,s_0-1]$ does not appear, because $\dot\phi=0$ on it. Hence, remembering that $\rho>0$ and $\Scal\geq 0$, the following estimate holds
\begin{eqnarray*}
\lambda\int_0^{s_0}\phi^2ds + \rho\int_0^{s_0}\Scal\phi^2ds-\int_0^{s_0}\phi^2\nabla_{\dot{\gamma}}\nabla_{\dot{\gamma}} f \, ds
& \geq & \lambda s_0-\frac43\lambda + 2\int_0^1\phi \, \nabla_{\dot{\gamma}} f  \, ds - 2\int_{s_0-1}^{s_0}\phi \, \nabla_{\dot{\gamma}} f \, ds \\
& \geq & \lambda s_0-\frac43\lambda - \max_{[0,1]}|\nabla_{\dot{\gamma}} f | - \max_{[s_0-1,s_0]}|\nabla_{\dot{\gamma}} f | \,.
\end{eqnarray*}
Combining this with inequality~\eqref{eq_snd_var}, we infer that
$$
\max_{[s_0-1,s_0]}|  \nabla_{\dot{\gamma}} f   |\geq\lambda s_0 - \frac43\lambda - 2n+2-c' \, ,
$$
where we set $c' = \max_{[0,1]}|\nabla_{\dot{\gamma}} f |$. Therefore, by Lemma \ref{lmm_grad_lin} and \eqref{eq_gradlip}, we obtain
$$
\dfrac{a^{\pm}}{2}+\sqrt{a^{\pm}f(q)+b^{\pm}}\geq\max_{[s_0-1,s_0]}|\nabla_{\dot{\gamma}} f|\geq \lambda s_0-c''=\lambda \dist_g(p,q)-c''\;,
$$
for some positive constant $c''$.


Suppose now that $f$ is not a constant function and let $\Sigma_0$ be as in the statement of the lemma. If we pick the point $p$ in $\Sigma_0$, the triangle inequality implies at once that, for every $q \in M$,
$$
\dfrac{a^{\pm}}{2}+\sqrt{a^{\pm}f(q)+b^{\pm}} \, \geq \, \lambda \, \dist_g(q,\Sigma_0) - c'' \, \geq \, \lambda \, | r (q)| - c'' \, ,
$$
where $r$ is the signed distance to $\Sigma_0$. With the usual abuse of notations, we can write 
$$
\sqrt{a^{\pm}f(r)+b^{\pm}}\, \geq \lambda \, |r| - c''' \, ,
$$ 
where $c'''=c''+\max\{a^+/2, a^-/2\}$.

Squaring this last inequality, we obtain
$$
f(r)\geq\dfrac{1}{a^{\pm}}\left((\lambda |r|-c''')^2-b^{\pm}\right)\;,
$$
which at once implies that $f$ is bounded from below. On the other hand, we observe that such a function is defined up to an additive constant. Thus, from now on, we will always assume $\min_M f  > 0$, without loss of generality; therefore, there exist positive constants $C$ and $D$ such that, for every admissible value of $r$,
$$
f (r) \, \geq \, C \, ( \,|r| - D\,)^2 \, .
$$
This concludes the proof of the lemma. 
\end{proof}

The lower bound on $f$ easily implies the following compact version of Lemma~\ref{lmm_grad_lin}.
\begin{cor}
\label{cor_bounds}
Let $(M^n, g,f)$ be a gradient shrinking $\rho$-Einstein soliton with $\rho>0$, $\RRR\geq 0$ and such that $|\RRR|\leq K$, for some positive constant $K$. Then, either $f$ is constant or there exist positive real constants $a$, $b$, $c$ and $d$, such that 
\begin{eqnarray*}
cf(r) - d \, \leq \, |\nabla f|^2(r) \, \leq \, af(r) + b\;, 
\end{eqnarray*}
where $r$ is the signed distance to a connected component $\Sigma_0\subset M$ of some regular level set of $f$. The constants which appear in the estimate are possibly depending on $\Sigma_0$.
\end{cor}
We are now in the position to prove the following upper bound for the potential function.
\begin{cor}
\label{cor_stima1}
Let $(M^n, g,f)$ be a gradient shrinking $\rho$-Einstein soliton with $\rho>0$, $\RRR\geq 0$ and such that $|\RRR|\leq K$, for some positive constant $K$. Then, either $f$ is constant on $M$ or there exist positive constants $A$ and $B$, such that 
$$
0\, < \, f (r) \, \leq \, A ( \,|r| +B \,)^2 \, ,
$$
where $r$ is the signed distance to a connected component $\Sigma_0\subset M$ of some regular level set of $f$. The constants which appear in the estimate are possibly depending on $\Sigma_0$.
\end{cor}
\begin{proof} If $f$ is not constant, by Corollary \ref{cor_bounds}, we have that
$$
\big|\nabla\sqrt{f} \big|^2 (r)  \, = \, \frac{1}{4}\frac{|\nabla f|^2 (r)}{f(r)} \, \leq \, \frac{1}{4}\left(a+\frac{b}{f(r)}\right) \,.
$$
Since we are assuming $\min f>0$, we deduce that $\sqrt{f}$ is a Lipschitz function. The conclusion follows at once.
\end{proof}

We conclude this section with the following proposition, which summarizes the results of Lemma \ref{lmm_stima2} and Corollary \ref{cor_stima1}.
\begin{prop}
\label{prp_quadratic}
Let $(M^n, g,f)$ be a gradient shrinking $\rho$-Einstein soliton with $\rho>0$, $\RRR\geq 0$ and such that $|\RRR|\leq K$, for some positive constant $K$. Then, either $f$ is constant on $M$ or there exist positive constants $A,B,C$ and $D$, such that 
$$
C\, (\,|r| - D \,)^2 \, \leq \, f (r) \, \leq \, A ( \,|r| + B \,)^2 \, ,
$$
where $r$ is the signed distance to a connected component $\Sigma_0\subset M$ of some regular level set of $f$. The constants which appear in the estimate are possibly depending on $\Sigma_0$.
\end{prop}

\

\section{Proof of Theorem~\ref{main}}

The aim of this section is to show that, under the assumption of Theorem~\ref{main}, the scalar curvature is a constant function. Since, by Theorem~\ref{teo_analytic}, the soliton metrics are real analytic for $\rho \notin \{1/n,1/2(n-1)\}$, it is sufficient to prove that $\RRR$ is constant on some open set. As we will see at the end of the section, this will imply Theorem~\ref{main}.   

From now on, we will assume that $(M^n,g,f)$ is a complete, non compact, gradient shrinking $\rho$-Einstein soliton with $0<\rho< 1/2(n-1)$, bounded curvature, nonnegative radial sectional curvature, and nonnegative Ricci curvature. 
We observe that under these assumptions, if the potential function $f$ were constant, then, by equation~\eqref{eq_soliton} and the Bonnet-Myers Theorem, the manifold would be compact, which is excluded. Hence, there has to exist a regular level set of $f$. Reasoning as in the previous section, we let $\Sigma_0 \subset M$ be a regular connected component of this level set and we have that $f$ only depends on the signed distance $r$ to $\Sigma_0$ on the whole manifold. With a small abuse of notation, we will consider $f$ as a function of $r$ and we will indicate by $f',f'',\ldots$ the derivatives of $f$ with respect to $r$.
As a consequence, we can express the gradient and the Hessian of $f$ as
\begin{eqnarray*}
\nabla f \,  =  \, f' \nabla r \, \quad & \hbox{and}& \quad
\nabla^2f  =  f'\nabla\nabla r + f'' dr \otimes dr\;.
\end{eqnarray*}

We observe that the signed distance $r$ must be unbounded on $M$. In fact, if this were not the case, by Proposition~\ref{prp_quadratic}, we would have that $f$ is bounded too. On the other hand, the Bakry-Emery Ricci tensor $\Ric + \nabla^2 f$ is bounded from below by $\lambda g$ and  this would imply that $M$ is compact, by~\cite[Theorem 1.4]{weywyl}.

As a first step we are going to prove that $f$ is a convex function on a set of the form $\{|r| \geq r_0\}$, for some $r_0>0$. Following Petersen-Wylie~\cite{pw2}, we are going to estimate the two terms of the Hessian separately. We start with the following lemma. 
%

\begin{lemma} 
\label{lmm_hessr}
Let $(M^n, g)$ be a complete noncompact Riemannian manifold
and let $\Sigma_0 \subset M$ be a regular hypersurface. We
denote by $r: M^n \rightarrow \RR$ the signed distance to $\Sigma_0$. If $\Rm(E, \nabla r , E , \nabla r) \geq 0$ for every $E\in T_p M$ which is orthogonal to $\nabla r$, then the following holds.
\begin{enumerate}
\item If $r$ is not bounded from above, then $\nabla^2 r$ is positive semi-definite in the region $\{ r > 0 \}$.
\item If $r$ is not bounded from below, then $\nabla^2 r$ is negative semi-definite in the region $\{r < 0 \}$.
\end{enumerate}
\end{lemma}
\begin{proof} 
We present the proof only in the first case, since the second one will follow by trivial adaptations. Let us set $S=\nabla^2 r$. As $|\nabla r|=1$, then
$$
0=\nabla_i(\nabla_j r \nabla_j r)=2\nabla_i\nabla_j r \nabla_j r=2S_{ij}\nabla_j r \, ,
$$
which implies
$$
0=\nabla_k(S_{ij}\nabla_j r)=\nabla_k S_{ij}\nabla_j r+S_{ij}S_{jk}\;.
$$
On the other hand it holds 
$$
\nabla_k S_{ij}-\nabla_i S_{kj}=(\nabla_k\nabla_i - \nabla_i\nabla_k)\nabla_jr \, =\, R_{kijl}\nabla_l r\;.
$$
Combining these identities, we get
\begin{equation}
\label{noidea}
\nabla_k S_{ij}\nabla_k r=\nabla_i S_{kj}\nabla_k r + R_{kijl}\nabla_lr\nabla_k r=-S_{jk}S_{ki}-\RRR_{ikjl}\nabla_lr\nabla_k r\;.
\end{equation}

We now let $\mu$ be the smallest eigenvalue of $S$. It is well known that $\mu$ is an absolutely continuous function. Therefore, it is weakly differentiable, its derivative is locally integrable and the integral along any curve of the derivative coincides almost everywhere with $\mu$. Moreover, it is differentiable almost everywhere. We want to compute $\langle \nabla\mu \,| \,\nabla r \rangle$, at a point $p$ where $\mu$ and $r$ are differentiable. We recall that the distance function to a submanifold is a Lipschitz function. In particular $r$ is absolutely continuous and differentiable almost everywhere. 
%
For $\eps>0$ sufficiently small, we let then $\gamma : (-\eps , \eps) \rightarrow M$ be an integral curve of $\nabla r$ with $\gamma(0)=p$, and we introduce the map 
$$
u \,: \, S_pM \times (-\eps, \eps) \, \longrightarrow \, \RR \, , \quad \quad (X,t) \, \longmapsto \, u(X,t) := S_{\gamma(t)}(\tilde{X}(t),  \tilde{X}(t) ) \, ,
$$ 
where $S_pM := \{ X \in T_p M \,\,: \,\, |X|^2 = 1 \, \}$ and $t \mapsto \tilde{X}(t)$ is the parallel transport of $X$ along $\gamma$, with the initial condition $\tilde{X}(0)=X$. Since, for every $t\in (-\eps, \eps)$, the parallel transport yields an isometry between $T_p M$ and $T_{\gamma(t)}M$, we have that $|\tilde{X}(t)|^2 \equiv 1$. It follows that 
$$
(\mu \circ \gamma) (t) \, = \, u_{\min} (t) \, : = \, \min_{X\in S_pM} \, u(X,t) \, .
$$
We observe that, with these definitions, one has $\langle \nabla\mu | \nabla r \rangle_p \, = \, \frac{d}{dt}\big|_{t=0} (\mu \circ \gamma ) \, = \,  \frac{d}{dt}\big|_{t=0} u_{\min}$. We claim that 
$$
\frac{d}{dt}{\Big|_{t=0}} u_{\min} \, = \, \frac{\partial u}{\partial t} \,(Y, 0) \, ,
$$
where $Y \in S_pM$ is such that $u\, (Y, 0) = u_{\min}(0)$. By Lagrange's Theorem, we have that for every $0<h<\eps$ there exists $\xi \in (0, h)$ such that 
$$
u_{\min}(h) \, \leq \, u\, (Y, h) \, = \, u\, (Y , 0) \, + \, h \, \frac{\partial u}{\partial t} \, (Y, \xi) \, = \, u_{\min}(0) \, + \,  h \, \frac{\partial u}{\partial t} \, (Y, \xi)\, .
$$
Dividing by $h$, subtracting $u_{\min}(0)$ from both sides and letting $h$ tend to zero, we obtain
$$
\frac{d}{dt}{\Big|_{t=0}} u_{\min} \, \leq \, \frac{\partial u}{\partial t} \,(Y, 0) \, ,
$$
since $\mu$ was differentiable at $p$. The other inequality is analogous and it is left to the reader. 

Having the claim at hand, we let $t \mapsto \tilde{Y}(t)$ be the parallel transport of $Y$ along $\gamma$ and we compute
$$
\langle \nabla\mu | \nabla r \rangle_p  \, = \, \frac{\partial u}{\partial t} \,(Y, 0) \, = \, \frac{d}{dt}{\Big|_{t=0}} S_{\gamma(\cdot)} \, (\tilde{Y} (\cdot), \tilde{Y} (\cdot) ) \, = \, \nabla_{\dot{\gamma}(0)} S \, (Y, Y) \,=\, \nabla_{\nabla r}  S \, (Y, Y)  \,.
$$
Using~\eqref{noidea}, we finally obtain that, at every point $p \in \{ r>0\}$ where $\mu$ and $r$ are differentiable, it holds 
$$
\langle \nabla \mu \vert\nabla r\rangle \, = \, \nabla_{\nabla r}  S \, (Y, Y)  \, = \, 
-S_{jk}S_{ki}Y_jY_i - \RRR_{ikjl}\nabla_lr \nabla_kr Y_i Y_j \, = \, -\mu^2|X|^2 - \RRR(Y,\nabla r, Y , \nabla r) \, .
$$




Since we are assuming that $\RRR(E, \nabla r , E , \nabla r) \geq 0$ for every $E\in T_p M$ which is orthogonal to $\nabla r$, we deduce, by the symmetries of the Riemann tensor, that $
\langle \nabla \mu\vert\nabla r\rangle \, \leq \,  -\mu^2$.

To complete the proof, we assume by contradiction that there exists $p_0 \in \{ r>0 \}$ such that $\mu(p_0) <0$ and we let $\a : [0, +\infty) \rightarrow M$ be an integral curve of $\nabla r$ with $\a(0)=p_0$. Notice that $\a$ is defined for every $t \geq 0$ because we are supposing that $r$ is not bounded from above (the variable $t$ differs by $r$ just by an additive constant, namely the distance between $p_0$ and $\Sigma_0$). By the absolute continuity of $\mu$, we have that $\mu(t)<0$, for every $t\geq 0$, since
$$
(\mu \circ \a) (t) \, \leq \, (\mu \circ \a) (0) \, - \int_{0}^t (\mu \circ \a)^2(s) \, ds \, .
$$
Hence, setting $w (t):= - 1/(\mu \circ \a)(t) >0$, we have that $\frac{d}{dt} w  \leq -1$, for almost every $t \geq 0$. Integrating from $0$ to $t$, we get 
$w(t) \, \leq \,w(0) \, - \, t$, which leads us to a contradiction, for large $t$. This completes the proof of the lemma.
%
\end{proof}

In the next proposition, we are going to prove that $f$ is convex at infinity.

\begin{prop}
\label{prp_fconvex} 
Let $(M^n,g,f)$ be a complete, non compact, gradient shrinking $\rho$-Einstein soliton with $0<\rho< 1/2(n-1)$, bounded curvature, nonnegative radial sectional curvature, and nonnegative Ricci curvature. Let $\Sigma_0 \subset M$ be a connected component of a regular level set of $f$ and let $r: M^n \rightarrow \RR$ be the signed distance to $\Sigma_0$. Then, the following holds.
\begin{enumerate}
\item If $r$ is not bounded from above, then there exists $r_0>0$, such that $\nabla^2 f$ is positive semi-definite in the region $\{ r \geq r_0 \}$.
\item If $r$ is not bounded from below, then there exists $r_0>0$, such that $\nabla^2 f$ is positive semi-definite in the region $\{r \leq -r_0 \}$. 
\end{enumerate}
\end{prop}

\begin{proof} 
We present the proof only in the first case, since the second one will follow by trivial adaptations.
By equations~\eqref{eq_soliton} and the expression of the Hessian of $f$, we have
$$
f'' \, = \, -\Ric(\nabla r,\nabla r) +\lambda +\rho \Scal\;.$$
We claim that $\Ric(\nabla r,\nabla r) \rightarrow 0$, as $r\rightarrow + \infty$. By Corollary~\ref{cor_bounds} and Proposition~\ref{prp_quadratic}, we have that $|\nabla f|^2 = (f')^2 \rightarrow + \infty$, as $r \rightarrow + \infty$. Thus, $f'(r)$ has a definite sign, provided $r$ is large enough. Again by Proposition~\ref{prp_quadratic}, we deduce that $f'(r)>0$, for large enough $r$. Thus, by~\eqref{equa2} in Lemma~\ref{lemg}, we get
$$
0\, \leq\, \Ric(\nabla r, \nabla r) \, = \, \frac{\Ric(\nabla f, \nabla f)}{(f')^2} \, = \, \frac{(1-2(n-1)\rho)}{2}\frac{\langle \nabla \Scal \, \vert \, \nabla f\rangle}{(f')^2} \, = \, \frac{(1-2(n-1)\rho)}{2} \, \frac{\RRR'}{f'}\, ,
$$
for $r$ large enough. To prove the claim, we assume by contradiction that 
$
\varlimsup_{r\rightarrow +\infty} (\RRR'/f') = \delta \, ,
$ for some $\delta>0$. On the other hand, we have that $\varliminf_{r\rightarrow + \infty}  (\RRR'/f')= 0$, since $\RRR$ is bounded and $\RRR' \geq 0$. In particular, there exist two sequences $(\hat{r}_k)_{k\in\mathbb{N}}$ and $(\check{r}_j)_{j\in \mathbb{N}}$ tending to infinity for $k,j \to +\infty$, such that 
$$
\lim_{k\to+\infty} (\Scal'/f') (\hat{r}_k) \, = \, \delta \,\,\quad \quad \hbox{and} \quad \quad\,\, \lim_{j\to+\infty} (\RRR'/f')(\check{r}_j) \, = \, 0 \, .
$$
Without loss of generality, we can assume that $(\RRR '/f')(\hat{r}_k) \, >\,\delta/2$ and  $(\RRR '/f')(\check{r}_j) \, <\,\delta/2$, for every $k,j\in \mathbb{N}$. We consider the following construction. We pick an element of the second sequence and we call it $\check{r}_{j_1}$. We then set ${k_1} := \min\,  \{ k \in \mathbb{N}  \, : \, \hat{r}_k \geq \check{r}_{j_1}  \}$. Then, by induction, we define
$j_i := \min\,  \{ j \in \mathbb{N}  \, : \, \check{r}_j \geq \hat{r}_{k_{i-1}}  \}$ and $k_i := \min\,  \{ k \in \mathbb{N}  \, : \, \hat{r}_k \geq \check{r}_{j_{i}}  \}$, for every $i\geq 1$. To fix the ideas, we observe that by construction one has that 
$\check{r}_{j_1} < \hat{r}_{k_1} < \check{r}_{j_2} < \hat{r}_{k_2} < \ldots$ and so on. It is now immediate to deduce that the function $(\Scal '/f')$ must attain a local interior maximum between $\check{r}_{j_i}$ and $\check{r}_{j_{i+1}}$, for every $i\in\mathbb{N}$. We then let $r_i$ be an interior maximum point for $(\Scal'/f')$ in $[\, \check{r}_{j_{i}} \, , \, \check{r}_{j_{i+1}}]$. Hence, we have obtained a sequence $(r_i)_{i\in\mathbb{N}}$ which tends to infinity, as $i\to +\infty$ and such that
$$
\lim_{i\to+\infty} (\Scal'/f') ({r}_i) \, = \, \delta \,\,\quad \quad \hbox{and} \quad \quad\,\, 0 \, = \, (\Scal'/f')' (r_i)\,= \, [ (\RRR''/f')  - (\RRR'/f')(f''/f')]\, (r_i) \, ,
$$
for every $i\in\mathbb{N}$. To find a contradiction, we are going to use equation~\eqref{equa3} in Lemma~\ref{lemg}, which in virtue of the rectifiability reads
$$
(1-2(n-1)\rho) \, [\,\RRR'' + \RRR' \Delta r\, ] \, = \,  \Scal' f' - 2|\Ric|^2+2\rho\Scal^2+2\lambda \Scal\;.
$$
As $\Ric \geq 0$, we have that $|\Ric|^2 \leq \RRR^2$. Therefore, since $\RRR$ is bounded, we have that $f''$ is bounded as well. By~\eqref{equa1} in Lemma~\ref{lemg} and the identity $\Delta f \, = \, f'' + f' \Delta r$, we deduce that $\Delta r \leq C$, for some positive constant $C>0$. Combining all these observations, we obtain that there exists a constant $K>0$ such that, at the $r_i$'s, we have
\begin{eqnarray*} 
0 &  = & (1-2(n-1)\rho) \, [\,  (\RRR''/f')  - (\RRR'/f')(f''/f') \,]  \\
& \geq & [\, f' - (1-2(n-1)\rho) \, C  - (1-2(n-1)\rho) (f''/f')\, ] \, (\RRR'/f') \, - \, (K/f') \\
& \geq & [\, f' - K - (K/f') \,  ] \, (\delta/2) \, - \, (K/f') \,.
\end{eqnarray*}
This contradicts the fact that $f'(r_i)\to +\infty$, for $i \to +\infty$ and the claim is proven. As a consequence, we have that $f''>0$, for $r$ large enough. Combining this with Lemma~\ref{lmm_hessr}, it is easy to deduce the statement of the proposition.
\end{proof}

We employ now the previous proposition to show that the scalar curvature is $f$-subharmonic at infinity. From this we deduce that $\RRR$ is actually constant on some open set. Hence, by analyticity, it must be constant everywhere. 

\begin{prop} 
\label{rconst}
Let $(M^n,g,f)$ be a complete, non compact, gradient shrinking $\rho$-Einstein soliton with $0<\rho< 1/2(n-1)$, bounded curvature, nonnegative radial sectional curvature, and nonnegative Ricci curvature. Then the scalar curvature $\Scal$ is constant.
\end{prop}

\begin{proof} As in the proof of Lemma~\ref{lmm_hessr} and Proposition~\ref{prp_fconvex}, we only consider the case where $r$, the signed distance to $\Sigma_0$, is not bounded from above. By Proposition \ref{prp_fconvex}, we have that, for $r \geq r_0$, the Hessian of $f$ is positive semi-definite. Hence, by equation~\eqref{eq_soliton}, 
$$
\Ric\leq \lambda g +\rho \Scal g \,.
$$
Writing equation~\eqref{equa3} as 
$$
(1-2(n-1)\rho) \, \Delta \Scal \, = \, \langle \nabla \Scal\vert\nabla f\rangle -2(\RRR_{ij}-\rho\Scal g_{ij}-\lambda g_{ij})\RRR_{ij}
$$
and noticing that, if $r$ is large enough, the term $-2(\RRR_{ij}-\rho\Scal g_{ij}-\lambda g_{ij})\RRR_{ij}$ is the product of two nonnegative quantities, we arrive to
$$
(1-2(n-1)\rho)\, [ \, \Delta \Scal-\langle\nabla\Scal\vert\nabla f\rangle \, ] \, \geq  \, 2(n-1)\rho \, \langle \nabla\Scal\vert\nabla f\rangle \, \geq \, 0 \, ,
$$ 
for $r\geq r_0$. So far, we have obtained that $\RRR$ is $f$-subharmonic at infinity, in the sense that 
$$
\Delta_f \RRR \geq 0 \, ,
$$ 
for $r\geq r_0$. Using the rectifiability, this condition reads
$$
\RRR'' + \RRR' \Delta r -  \Scal' f' \, \geq \, 0 \, .
$$
As we noticed in the proof of the previous proposition, under our assumptions we have $|\Ric|^2 \leq \RRR^2$ and $\Delta r \leq C$, for some positive constant $C>0$.
Combining this with Lemma~\ref{lmm_grad_lin} and Proposition~\ref{prp_quadratic}, we deduce that there exists a real number $r_1>0$ such that 
$$
\RRR'' \, \geq \, [\,f' - C\, ] \, \RRR' \, \geq \, 0 \, ,
$$
for $r\geq r_1$. In particular, $\RRR'(r) \geq \RRR'(r_1) \geq 0$, for every $r\geq r_1$. Integrating $\RRR'$, we get
$$
\RRR(r) \, = \, \RRR(r_1) \, + \int_{r_1}^r \RRR'(s) \, ds \, \geq \, \RRR(r_1) \, +  \RRR'(r_1) \, (r-r_1) \, .
$$
Since $\RRR$ is bounded, the only possibility is that $\RRR'(r_1) = 0$. Replying this argument for every $r_2 \geq r_1$, we deduce that $\RRR$ is constant in the region $\{\, r \geq r_1\,\}$. By the analyticity of $\RRR$, see Theorem~\ref{teo_analytic}, we conclude that $\RRR$ must be constant everywhere.
\end{proof}
We observe now that the previous proposition combined with Theorem~\ref{teorect} implies that, for $0<\rho< 1/2(n-1)$, our $\rho$-Einstein soliton is actually a {rectifiable} gradient shrinking Ricci soliton satisfying all the assumptions in Theorem~\ref{petwyl} in~\cite{pw2}. Hence, it is rigid and the proof of Theorem~\ref{main} is complete.

\

\begin{ackn} The authors are partially supported by the Italian projects FIRB--IDEAS ``Analysis and Beyond'' and GNAMPA ``Flussi geometrici e solution autosimilari''.
\end{ackn}

\bibliographystyle{amsplain}
\bibliography{rigid}

\end{document}